\documentclass[10pt, letterpaper]{amsart}
\usepackage{amsmath,amsfonts,amssymb, tikz, amsthm}
\usepackage[colorlinks=true,linkcolor=black,anchorcolor=black,citecolor=black,filecolor=black,menucolor=black,runcolor=black,urlcolor=blue]{hyperref}
\usepackage[utf8]{inputenc}
\usepackage[english]{babel}
\usepackage{xcolor}
\usepackage[tableposition=top]{caption} 
\usepackage{epigraph}
\usepackage{mathtools}
\usepackage{musicography}

\newtheorem*{theorem*}{Theorem}
\newtheorem*{proposition*}{Proposition}
\newtheorem{theorem}{Theorem}
\newtheorem{lemma}{Lemma}
\newtheorem{proposition}{Proposition}

\newtheorem{question}{Question}
\newtheorem{conjecture}{Conjecture}

\theoremstyle{remark}
\newtheorem*{remark*}{Remark}

\newtheorem{remark}{Remark}

\theoremstyle{proof}

\numberwithin{equation}{section}

\Alph{assumption}

\newcommand{\Z}{\mathbb{Z}}

\newcommand{\C}{\mathbb{C}}

\newcommand{\N}{\mathbb{N}}

\begin{document}
	
	\title{On a conjecture of Corr\'adi and K\'atai}
	
	\author{Krishnarjun Krishnamoorthy}
	\email[Krishnarjun Krishnamoorthy]{krishnarjunmaths@outlook.com, krishnarjun@nitt.edu}
	\address{Department of Mathematics, National Institute of Technology Tiruchirappalli, Trichy, Tamil Nadu, 620015, India}
	\email[Krishnarjun Krishnamoorthy]{krishnarjunmaths@outlook.com}
	
	
	\maketitle

	\begin{abstract}
		We consider Goldbach type sums corresponding to the Liouville function and prove the existence of sufficient cancellations. We also consider applications to sign patterns in the Liouville function.
	\end{abstract}
	
	\section{Introduction}\label{Section "Introduction"}
	
	Many interesting number theoretic information are deduced from the study of various averages of arithmetic functions. For instance, the prime number theorem is well known to be equivalent to the asymptotic estimate,
	\[
	\sum_{n\leqslant X} \Lambda(n) \sim  X,
	\]
	where $\Lambda(n)$ is the von Mangoldt function given by
	\[
	\Lambda(n) = \begin{cases}
		\log (p) &\mbox{if } n=p^r\mbox{ for some prime }p\\
		0 & \mbox{otherwise}.
	\end{cases}
	\]
	 A more challenging type of sums involving the von Mangoldt function is the following. Suppose that $N$ is a given integer, then we may consider the sum
	\begin{equation}\label{Equation "Goldbach von Magoldt"}
		\sum_{n=1}^{N-1} \Lambda(n) \Lambda(N-n) = \sum_{a+b=N} \Lambda(a)\Lambda(b).
	\end{equation}
	
	Finding strong lower bounds for the sum given in \eqref{Equation "Goldbach von Magoldt"} as $N\to\infty$ will lead to progress towards the \textit{strong Goldbach's conjecture} which states that every even number is the sum of two prime numbers. The sum of $a,b$ above runs over positive integers. We shall have no instance to consider non-positive integers throughout this paper.
	
	The purpose of this note is to study the analogue of \eqref{Equation "Goldbach von Magoldt"} and its higher rank analogues with the von Mangoldt function replaced by the Liouville function. Let $\lambda(n)$ denote the Liouville function; defined by the relations $\lambda(mn) = \lambda(m)\lambda(n)$ and $\lambda(p)=-1$ for any two integers $m,n$ and any prime $p$. Let $k,N$ be given positive integers. We define the set 
	\begin{equation}\label{Equation "S_{k} definition"}
		\mathcal{S}_{k}(N) := \left\{(a_1,\ldots,a_k)\in \N^{k}\ |\ a_1+\ldots+a_k = N\right\}.
	\end{equation}
	Since $|\mathcal{S}_{2}(N)| = N-1$ and
	\[
	\left|\mathcal{S}_{k}(N)\right| = \sum_{a=1}^{N-1} \left|\mathcal{S}_{k-1}(N-a)\right|,
	\]
	we may calculate, for example by induction on $k$, that
	\[
	\left|\mathcal{S}_{k}(N)\right| = \binom{N-1}{k-1}
	\]
	as $N\to\infty$ and for a fixed $k$. Throughout the paper we adopt the convention that $\binom{k}{\ell} = 0$ if $k < \ell$. Define the Goldbach-type sum
	\begin{equation}\label{Equation "G_{k} definition"}
		\mathcal{G}_{k}(N) := \underset{(a_1,\ldots,a_{k})\in \mathcal{S}_{k}(N)}{\sum \ldots \sum} \lambda(a_1)\ldots\lambda(a_k).
	\end{equation}
	In this paper we shall study $\mathcal{G}_{k}(N)$ for $k\geqslant 2$ as $N\to\infty$. We wish to improve upon the trivial bound 
	\begin{equation}\label{Equation "Trivial bound"}
		\left|\mathcal{G}_{k}(N)\right|\leqslant \left|\mathcal{S}_{k}(N)\right|\sim \frac{N^{k-1}}{(k-1)!}.
	\end{equation}

	\subsection{Main results}\label{Subsection "Main results"}
	
	First, we state our results for $k\geqslant 3$.
	
	\begin{theorem}\label{Theorem "G_{k}, k>=3"}
		For any positive constant $A$ and for any fixed integer $k\geqslant 3$, as $N\to\infty$ we have 
		\[
		\left|\mathcal{G}_{k}(N)\right| \ll_{A,k} \frac{N^{k-1}}{\log^A(N)}.
		\]
	\end{theorem}
	Thus, we have obtained an improvement over the trivial bound \eqref{Equation "Trivial bound"} as long as $k\geqslant 3$ is fixed and $N\to\infty$. This leaves the case when $k=2$. This case is much trickier to handle and we obtain only a partial result.  For every $\delta > 0$, let us define
	\begin{equation}\label{Equation "K(epsilon) definition"}
		\mathcal{K}(\delta) := \left\{N \in \N\ |\ |\mathcal{G}_{2}(N)| > \delta N\right\}.
	\end{equation}
	Our next theorem addresses the case when $k=2$.
	\begin{theorem}\label{Theorem "K(delta) bound"}
		For every $\delta > 0$,
		\[
		\left|\mathcal{K}(\delta)\cap \{1,2,\ldots,N\}\right| \ll_{A} \frac{1}{\delta^2} \frac{N}{\log^A(N)}
		\]
		for any positive constant $A$. In particular,
		\[
		\liminf_{N\to\infty} \frac{\left|\mathcal{G}_{2}(N)\right|}{N} = 0.
		\]
	\end{theorem}
	It is conjectured that $\left|\mathcal{G}_{2}(N)\right|=o(N)$ (see Conjecture \ref{Conjecture "Corradi-Katai"} below). We also present an interesting corollary of Theorem \ref{Theorem "G_{k}, k>=3"} regarding the equi-distribution of signs of the Liouville function. In fact, it can be gleaned from the proof that Theorem \ref{Theorem "Sign pattern"} and Theorem \ref{Theorem "G_{k}, k>=3"} are equivalent.

	\begin{theorem}\label{Theorem "Sign pattern"}
		Let $k\geqslant 3$ be any integer, and let $(\varepsilon_1,\ldots,\varepsilon_{k})\in \{\pm 1\}^{k}$ denote an arbitrary sequence of ``signs''. Then
		\[
		\left|\left\{ (a_1,\ldots,a_{k})\in \mathcal{S}_{k}(N)\ |\  (\lambda(a_1),\ldots,\lambda(a_{k})) = (\varepsilon_1,\ldots,\varepsilon_{k})\right\}\right| = \frac{1}{2^{k}} \binom{N^{k-1}}{k-1} + \mathcal{O}_{A} \left(\frac{N^{k-1}}{\log^{A}(N)}\right)
		\]
		for a fixed $k$ and as $N\to\infty$.
	\end{theorem}

	\subsection{Previous works}\label{Subsection "Previous works"}

	In 1969, Corr\'adi and K\'atai made the following conjecture \cite{CK}.
	\begin{conjecture}[Corr\'adi-K\'atai]\label{Conjecture "Corradi-Katai"}
		With notation as above,
	\begin{equation}\label{Equation "Conjecture"}
			\lim_{N\to\infty} \frac{\left|\mathcal{G}_2(N)\right|}{N} = 0.
	\end{equation}
	\end{conjecture}
	Thus Theorem \ref{Theorem "K(delta) bound"} asserts that Conjecture \ref{Conjecture "Corradi-Katai"} is true, should the limit defining \eqref{Equation "Conjecture"} exist. Conjecture \ref{Conjecture "Corradi-Katai"} was proven conditionally (on the existence of infinitely many Siegel zeros) in \cite{CKProof} but has resisted an unconditional proof so far. In fact the authors of \cite{CKProof} prove a superior power saving error term. Recently, Sarnak asked the following substantially weaker question.
	\begin{question}[Sarnak]\label{Question "Sarnak"}
		Is it true that $|\mathcal{G}_{2}(N)| < N-1$ for all $N \gg 1$?
	\end{question}
	Question \ref{Question "Sarnak"} asks for \textit{any} improvement upon the trivial bound \eqref{Equation "Trivial bound"} and was only recently answered in the positive by Mangerel \cite{MangerelIMRN}.
	\begin{theorem}[Mangerel]\label{Theorem "Mangerel 1"}
		If $N\notin\{2,3,5,10\}$, then $|\mathcal{G}_{2}(N)| < N-1$.
	\end{theorem}
	Define $\phi(N)$ so that $|\mathcal{G}_{2}(N)| = N-1 - \phi(N)$. It follows that $\phi(N) \geqslant 0$. As progress towards Conjecture \ref{Conjecture "Corradi-Katai"}, it is important to understand the true order of growth\footnote{Assuming a strong zero-free region for the Dirichlet $L$ functions, a lower bound for $\phi(N)$ is obtained in \cite{mangerel2024shustermansgoldbachtypeproblemsign}.} of $\phi(N)$.

	The proof of Theorem \ref{Theorem "Mangerel 1"} involved first a reduction argument, which reduced to the case when $N$ was a prime\footnote{A similar reduction argument was used in \cite{KKPAMS} to obtain combinatorial proofs of various results associated to the Chowla's conjecture.}. Supposing the contradiction for a prime $p$, one deduces that the Liouville function mimics the quadratic residue symbol modulo $p$, almost exactly. Leveraging the properties of Gauss sums derives a contradiction.

	\subsection*{Acknowledgement}
	
	The author wishes to thank Prof. Akhilesh Parol for inviting the author to visit the Kerala School of Mathematics. The author also wishes to thank the institute for its generous hospitality and wonderful research atmosphere.	
	
	\section{Preliminaries}
	
	\subsection{Fourier theory}
	
	Suppose $\{a_{n}\}\in \ell^{2}(\C)$ is an element of the Hilbert space of all square integrable sequences of complex numbers. Then we know that $T(x) := \sum_{n=1}^{\infty} a_{n} e(nx) \in L^2([0,1])$. Here and henceforth, we follow standard conventions and denote by $e(x)$, the quantity $e^{2\pi ix}$. A simple calculation proves the following lemma.
	
	\begin{lemma}\label{Lemma "Fourier Coefficient"}
		Let $N\in \Z$ be a given integer. The following hold true.
		\begin{enumerate}
			\item	\[
			\int\limits_{0}^{1} T^2(x) e(-Nx) dx = \sum_{n_1 + n_2 = N} a_{n_1} a_{n_2},
			\]
			\item 	\[
			\int\limits_{0}^{1} |T(x)|^2 e(-hx) dx = \sum_{n_1 - n_2 = h} a_{n_1} \overline{a_{n_2}}.
			\]
		\end{enumerate}
	\end{lemma}
	
	\begin{proposition}\label{Proposition "Two norm average"}
		As $N\to\infty$, we have
		\begin{equation}\label{Equation "Two norm average"}
			\sum_{n=1}^{N} \left|\frac{\mathcal{G}_{2}(n)}{n}\right|^2 \ll_{A} \frac{N}{\log^{A}(N)}.
		\end{equation}
	\end{proposition}
	
		The proof of Proposition \ref{Proposition "Two norm average"}, relies on a result of Davenport. To describe this, we introduce the following ($L^2$ normalized) exponential sum
	\begin{equation}\label{Equation "S(N) definition"}
		S(N,x):= \frac{1}{\sqrt{N}} \sum_{n=1}^{N} \lambda(n) e(nx).
	\end{equation}
	Then Davenport's theorem \cite{Davenport} is the following.
	\begin{theorem}[Davenport]\label{Theorem "Davenport"}
		Let $S(N,x)$ be as in \eqref{Equation "S(N) definition"}. The 
		\[
		\left|S(N,x)\right| \ll_{A} \frac{\sqrt{N}}{\log^A(N)}
		\]
		for any positive constant $A$ and uniformly for $x\in [0,1]$.
	\end{theorem}
	
	\begin{remark}
		We shall use the letter $A$ to denote an arbitrary constant. The exact value of $A$ maybe different at different occurrences.
	\end{remark}

	First we have the following lemma.
	
	\begin{lemma}\label{Lemma "4 norm"}
		We have
		\[
		\int\limits_{0}^{1} \left|S(N,x)\right|^4 dx\ll_{A} \frac{N}{\log^A(N)}.
		\]
	\end{lemma}
	
	\begin{proof}
		By direct computation, we have 
		\begin{equation}\label{Equation "2 norm"}
			\int\limits_{0}^{1} \left|S(N,x)\right|^{2} dx = 1.
		\end{equation}
		Comparing this with Theorem \ref{Theorem "Davenport"}, we have 
		\[
		\int\limits_{0}^{1} |S(N,x)|^4dx \ll_{A} \frac{N}{\log^A(N)} \int\limits_{0}^{1} |S(N,x)|^2 dx \ll_{A} \frac{N}{\log^A(N)}.
		\]
		This completes the proof of Lemma \ref{Lemma "4 norm"}.
	\end{proof}

	Applying Lemma \ref{Lemma "Fourier Coefficient"} to $S(N,x)$ for $n\leqslant N$, we have
	\[
	\int\limits_{0}^{1} S^{2}(N,x) e(-nx) dx = \frac{\mathcal{G}_{2}(n)}{N}
	\]
	for any integer $N$. 
	Similarly,
	\begin{equation}\label{Equation "4 norm 2"}
		\int\limits_{0}^{1} \left|S(N,x)\right|^4 dx =\int\limits_{0}^{1} \left|S^2(N,x)\right|^2dx
		= \frac{1}{N^2} \sum_{n=0}^{N} \left|\mathcal{G}_{2} (n)\right|^2.
	\end{equation}
	
	\begin{proof}[Proof of Proposition \ref{Proposition "Two norm average"}]
		
		From Lemma \ref{Lemma "4 norm"} and \eqref{Equation "4 norm 2"}, we have 
		\begin{equation}\label{Equation "2 norm bound"}
			\frac{1}{N^2} \sum_{n=1}^{N} \left|\mathcal{G}_{2}(n)\right|^2 \ll_{A} \frac{N}{\log^A(N)}.
		\end{equation}
		Partial summation gives us 
		\begin{align*}
			\sum_{n=1}^{N}\frac{\left|\mathcal{G}_{2}(n)\right|^2}{n^2} &= \frac{1}{N^2} \sum_{n=1}^{N} \left|\mathcal{G}_{2}(n)\right|^2 - 1 + 2\int\limits_{1}^{N} \frac{1}{t^3} \left(\sum_{n=1}^{t} \left|\mathcal{G}_{2}(n)\right|^2 \right)dt\\
			&= \frac{1}{N^2} \sum_{n=1}^{N} \left|\mathcal{G}_{2}(n)\right|^2 - 1 + 2\left( \int\limits_{1}^{\sqrt{N}} \frac{1}{t^3} \left(\sum_{n=1}^{t} \left|\mathcal{G}_{2}(n)\right|^2 \right)dt + \int\limits_{\sqrt{N}}^{N} \frac{1}{t^3} \left(\sum_{n=1}^{t} \left|\mathcal{G}_{2}(n)\right|^2 \right)dt\right).
		\end{align*}
		Bounding trivially (using $|\mathcal{G}_{2}(n)| \leqslant n-1$), the first integral is $\mathcal{O}(\sqrt{N})$. From \eqref{Equation "2 norm bound"}, the second term is $\ll_{A} N \log^{-A}(N)$. Thus
		\[
		\sum_{n=1}^{N}\frac{\left|\mathcal{G}_{2}(n)\right|^2}{n^2}  = \frac{1}{N^2} \sum_{n=1}^{N} \left|\mathcal{G}_{2}(n)\right|^2 + \mathcal{O}_{A}\left(\frac{N}{\log^A(N)}\right).
		\]
		Combining this with \eqref{Equation "2 norm bound"} completes the proof.
	\end{proof}
	
	\section{Proofs of main theorems}\label{Section "Proofs of main theorems"}
	
	\subsection{Proof of Theorem \ref{Theorem "G_{k}, k>=3"}}
	
	We begin with the proof of Theorem \ref{Theorem "G_{k}, k>=3"}. The proof is by induction on $k$. The base step is when $k=3$, which we prove first. From Proposition \ref{Proposition "Two norm average"}, we have,
	\[
	\sum_{n=1}^{N} \left|\mathcal{G}_{2}(n)\right|^2 \ll_{A} \frac{N^3}{\log^A(N)}.
	\]
	Hence,
	\[
	\mathcal{G}_{3}(N) = \sum_{n=1}^{N-1} \lambda(n) \mathcal{G}_{2}(N-n) \ll_{A} \frac{N^2}{\log^A(N)}
	\]
	from Cauchy-Schwartz inequality.  Let us fix an integer $k_{0} \geqslant 4$. Suppose the theorem is true for all $k=k_{0}-1$. Then, on bounding trivially, we have
	\begin{equation}\label{Equation "G_k recursion"}
		\mathcal{G}_{k_0}(N) = \sum_{a_{1}=1}^{N-k_{0}+1} \lambda(a_{1}) \mathcal{G}_{k_{0}-1}(N-a_{1}) \ll_{A} \frac{N^{k_{0}-1}}{\log^A(N)}
	\end{equation}
	as claimed. This completes the proof of Theorem \ref{Theorem "G_{k}, k>=3"}.

	\subsection{Proof of Theorem \ref{Theorem "K(delta) bound"}}
	
	Now we prove Theorem \ref{Theorem "K(delta) bound"}. We start with Proposition \ref{Proposition "Two norm average"}. Since the terms on the left hand side of \eqref{Equation "Two norm average"} are non-negative, we may drop some of the terms and have
	\[
	\left|\mathcal{K}(\delta)\cap [N]\right|\cdot \delta^2 \leqslant \sum_{n\in \mathcal{K}(\delta)\cap [N]} \left|\frac{\mathcal{G}_{2}(n)}{n}\right|^2 \ll_{A} \frac{N}{\log^A(N)}
	\]
	as claimed.

	\subsection{Proof of Theorem \ref{Theorem "Sign pattern"}}
	
	Suppose $\underline{\varepsilon} = (\varepsilon_{1}, \ldots, \varepsilon_{k})\in \{\pm 1\}^{k}$ is a choice of signs. Let
	\begin{equation}
		\mathcal{T}_{\underline{\varepsilon}}(N) := \left|\left\{ (a_1,\ldots,a_{k})\in \mathcal{S}_{k}(N)\ |\  (\lambda(a_1),\ldots,\lambda(a_{k})) = (\varepsilon_1,\ldots,\varepsilon_{k})\right\}\right|.
	\end{equation}	
	We shall show that 
	\begin{equation}\label{Equation "To Show"}
		\mathcal{T}_{\underline{\varepsilon}}(N) = \frac{1}{2^{k}} \binom{N^{k-1}}{k-1} + \left(\prod_{i=1}^{k} \varepsilon_{i}\right) \mathcal{G}_{k}(N) + \mathcal{O}_{A}\left(\frac{N^{k-1}}{\log^A(N)}\right).
	\end{equation}
	Theorem \ref{Theorem "Sign pattern"} follows from \eqref{Equation "To Show"} and Theorem \ref{Theorem "G_{k}, k>=3"}.

	The proof of \eqref{Equation "To Show"} is by induction on $k$. As the base step of induction, let us consider the case when $k=2$. Let $\underline{\varepsilon} = (\varepsilon_{1}, \varepsilon_{2})\in \{\pm 1\}^{2}$ be a choice of signs. We are interested in evaluating
	\[
	\mathcal{T}_{\underline{\varepsilon}} (N) := \frac{1}{4} \sum_{n_1 + n_2 = N} (1+\varepsilon_{1} \lambda(n_1)) (1+ \varepsilon_{2} \lambda(n_2))
	= \frac{N-1}{4} + \frac{(\varepsilon_{1} + \varepsilon_{2})}{4} \sum_{n=1}^{N-1} \lambda(n) + \frac{1}{4} \varepsilon_{1} \varepsilon_{2} \mathcal{G}_{2}(N).
	\]
	Hence \eqref{Equation "To Show"} with $k=2$ follows from here since the second term on the right hand side above is $\mathcal{O}_{A}(N/\log^A(N))$ by the prime number theorem ($x=0$ case of Theorem \ref{Theorem "Davenport"}).
	
	Suppose $k > 2$ is an integer and that \eqref{Equation "To Show"} is true for all integers upto $k-1$. Let $\underline{\varepsilon} = (\varepsilon_{1}, \ldots, \varepsilon_{k})\in \{\pm 1\}^{k}$ be as above and define $\underline{\varepsilon}':= (\varepsilon_{2}, \ldots, \varepsilon_{k})$. We have 
	\begin{align*}
		\mathcal{T}_{\underline{\varepsilon}} (N) &:= \frac{1}{2} \sum_{n = 1}^{N-1} (1+\varepsilon_{1} \lambda(n))\mathcal{T}_{\underline{\varepsilon}'}(N-n)\\
		&= \frac{1}{2} \sum_{n = 1}^{N-1} (1+\varepsilon_{1} \lambda(n))\left(\frac{1}{2^{k-1}} \binom{N-n-1}{k-2} + \frac{1}{2^{k-1}}\left(\prod_{i=2}^{k} \varepsilon_{i}\right)\mathcal{G}_{k-1}(N-n) + \mathcal{O}_{A}\left(\frac{N^{k-2}}{\log^A(N)}\right)\right).
	\end{align*}
	Rearranging the above gives us
	\begin{multline*}
		\mathcal{T}_{\underline{\varepsilon}}(N) = \frac{1}{2^{k}} \sum_{n=1}^{N-1} \binom{N-n-1}{k-2} + \frac{1}{2^k}\left(\prod_{i=2}^{k} \varepsilon_{i}\right) \sum_{n = 1}^{N-1}\mathcal{G}_{k-1}(N-n) + \frac{\varepsilon_{1}}{2^{k}} \sum_{n=1}^{N-1} \lambda(n) \binom{N-n-1}{k-2} + \\
		\frac{1}{2^{k}}\left(\prod_{i=1}^{k} \varepsilon_{i}\right) \sum_{n=1}^{N-1} \lambda(n) \mathcal{G}_{k-1}(N-n) + \mathcal{O}_{A} \left(\frac{N^{k-1}}{\log^A(N)}\right).
	\end{multline*}
	The first term on the right equals the main term $\frac{1}{2^k} \binom{N-1}{k-1}$. The last term equals $	\frac{1}{2^{k}}\left(\prod_{i=1}^{k} \varepsilon_{i}\right) \mathcal{G}_{k}(N) $ from \eqref{Equation "G_k recursion"}. The second term is $\mathcal{O}_{A}\left(N^{k-1} \log^{-A}(N)\right)$. For $k > 3$, this follows from Theorem \ref{Theorem "G_{k}, k>=3"}. For $k=3$, this follows from \eqref{Equation "2 norm bound"} and the Cauchy-Schwartz inequality. The remaining term can be bounded as follows. For simplicity, by a slight abuse of notation, let us denote $\sqrt{N}S(N,0)$ as simply $S(N)$. We define $S(0)$ as $0$.
	Then we have 
	\begin{multline*}
		\sum_{n=1}^{N-1} \lambda(n) \binom{N-n-1}{k-2} = \sum_{n=1}^{N-1}  \binom{N-n-1}{k-2} \left(S(n) - S(n-1)\right) = \sum_{n=1}^{N-1} S(n) \left(\binom{N-n-1}{k-2} - \binom{N-n}{k-2}\right)\\
		= - \sum_{n=1}^{N-1} S(n) \binom{N-n-1}{k-3} + S(N-1)\binom{1}{k-2}.
	\end{multline*}
	The last term occurs only if $k=3$. It is smaller in size than the error we are claiming and hence maybe ignored. In any case, the right hand side is bounded by
	\[
	\ll \max_{n\leqslant N} \left|S(n)\right| \sum_{n=1}^{N-1} \binom{N-n-1}{k-3} \ll_{k,A} \frac{N^{k-1}}{\log^A(N)} .
	\]
	This completes the proof.
	
	\bibliographystyle{plain}
	
	\bibliography{Bibliography}

\end{document}